\documentclass[11pt,a4paper]{amsart}

\usepackage[T1]{fontenc}
\usepackage[utf8]{inputenc}
\usepackage{amsmath,amssymb,amsthm}
\usepackage{mathrsfs}
\usepackage{tikz-cd}
\usepackage[all]{xy}
\usepackage[colorlinks=true,linkcolor=blue,citecolor=blue,urlcolor=blue]{hyperref}

\newtheorem{theorem}{Theorem}[section]
\newtheorem{proposition}[theorem]{Proposition}
\newtheorem{lemma}[theorem]{Lemma}

\theoremstyle{definition}
\newtheorem{definition}[theorem]{Definition}
\theoremstyle{remark}
\newtheorem{remark}[theorem]{Remark}

\newcommand{\C}{\mathcal{C}}
\newcommand{\Nil}{\mathsf{Nil}}
\newcommand{\Grp}{\mathsf{Grp}}
\newcommand{\Ab}{\mathsf{Ab}}

\newcommand{\Ker}{\operatorname{Ker}}
\newcommand{\Eq}{\operatorname{Eq}}
\newcommand{\Pt}{\mathsf{Pt}}

\title[Central extensions relative to $n$-nilpotent objects]{A characterization of central extensions relative to n-nilpotent objects}
\author{Marino Gran}
\address{Institut de Recherche en Math\'ematique et Physique,
Universit\'e catholique de Louvain,
Chemin du Cyclotron 2, 1348 Louvain-la-Neuve, Belgium}
\email{marino.gran@uclouvain.be}

\author{James R.\ A.\ Gray}
\address{Mathematics Division, Department of Mathematical Sciences,
Stellenbosch University, Private Bag X1, Matieland 7602, South Africa}
\email{jamesgray@sun.ac.za}

\author{Vito Volpe}
\address{Dipartimento di Matematica ``Federigo Enriques'',
Universit\`a degli Studi di Milano,
Via Saldini 50, 20133 Milano, Italy}
\email{vito.volpe@unimi.it}

\date{\today}
\thanks{The first author's research was supported by the Fonds de la Recherche
Scientifique -- FNRS under Grant CDR No.~J.0092.26.
The first author made use of Claude Fable 5 (Anthropic) in the
preparation of this paper: the AI assistant was useful to improve the overall
presentation and exposition. The ideas and the proofs in the paper are due to the authors. The authors verified all AI-assisted material
and take full responsibility for the final text.
}

\begin{document}

\begin{abstract}
We characterize the central extensions, in the sense of the categorical
theory of central extensions of Janelidze and Kelly, relative to the
Birkhoff subcategory of $n$-nilpotent objects of a
semi-abelian and algebraically coherent category. An extension $f \colon A \to B$ with kernel $K$ turns
out to be central precisely when the iterated Huq commutator
$[A,[A,\dots,[A,K]\dots]]$, with $n$ occurrences of $A$, vanishes. For $n=2$ this condition reads $[A,[K,A]] = 0$,
that is, $[K,A]$ is a central subobject of
$A$. In the special case of groups this description was obtained by Everaert and
the first author. The present proof is categorical, and applies in particular
to Lie algebras, to crossed modules, to cocommutative Hopf algebras over
a field, and to any category of interest.
We also show that, in an arbitrary semi-abelian category, an extension
$f \colon A \to B$ with kernel $K$ is central relative to the Birkhoff subcategory of
two-nilpotent objects, defined by means of the ternary Higgins commutator, precisely
when the ternary Higgins commutator $[K,A,A]_{\mathsf{Hig}}$ vanishes; a similar
characterization holds in the $n$-nilpotent case.
\end{abstract}

\maketitle

\section{Introduction}
A special case of categorical Galois theory \cite{Jan90}, which has found many interesting applications in algebra, is
the categorical theory of central extensions of Janelidze and
Kelly~\cite{JK94,JK00}. Most often, the starting point is a Birkhoff subcategory $\mathcal{X}$ of an exact Mal'tsev category $\C$ (a full reflective subcategory closed in $\C$ under subobjects and regular quotients), and the fundamental theorem of Galois theory
provides a complete classification of the so-called
$\mathcal{X}$-central extensions, defined categorically with respect to
 the reflector $I \colon \C \to \mathcal{X}$. When $\mathcal{X}$
is the subcategory $\Ab(\C)$ of abelian objects of a semi-abelian
category $\C$ (in the sense of Janelidze, M\'arki and Tholen \cite{JMT}), one recovers the classical notion: an extension
$f \colon A \to B$ with kernel $K$ is central exactly when $[K,A] = 0$
for the Huq commutator~\cite{JK00, Gran-Applications}. Determining, for other Birkhoff
subcategories, an explicit commutator condition equivalent to relative
centrality is a delicate problem, in general.

For the varieties $\Nil_n$ of $n$-nilpotent groups and $\mathsf{Sol}_n$
of $n$-solvable groups this problem was solved by Everaert and the first
author in~\cite{EG}, where the relative commutator of Everaert~\cite{E}
associated with these subvarieties was computed explicitly, and the
corresponding relative central extensions were described. The arguments
of~\cite{EG} are, however, specific to groups: they rely on the
description of the relative commutator by means of generators, hence on the
possibility of computing with elements and with iterated group-theoretic
commutator words. In particular, they do not extend to Lie algebras, cocommutative Hopf algebras,
crossed modules, or to the other algebraic structures to which the
semi-abelian approach applies, and no proof valid beyond the group case
appears to have been available in the literature.

The purpose of this paper is to give such a proof for the Birkhoff subcategory of $n$-nilpotent objects, for any $n \geq 1$. We first treat in detail the case $n = 2$, which already displays all the ideas of the general argument.
Working in a semi-abelian category which is algebraically coherent in the
sense of Cigoli, Gray and Van der Linden~\cite{CGVdL}, we show that an
extension $f \colon A \to B$ is
central relative to the Birkhoff subcategory $\Nil_2(\C)$ of
$2$-nilpotent objects precisely when its kernel $K$ satisfies
\[
  [A,[K,A]] \;=\; 0 ,
\]
the commutators being those of Huq (Theorem~\ref{thm:main}).
Two ingredients play a key role  in the proof of this result. The first
is the \emph{Three Subobjects Lemma} (Lemma~\ref{lem:three}), which
holds in any algebraically coherent semi-abelian category~\cite{CGVdL}
and yields the inequality $[K,[A,A]] \leq [A,[K,A]]$. The second is  the
\emph{Smith is Huq} condition \cite{BG02, MFVdL}, also available under algebraic coherence,
which allows one to translate the resulting vanishing of the Huq
commutator $[K,[A,A]]$ into the existence of a centralizing double
equivalence relation $C$ on the kernel pair $\Eq(f)$ of $f$ and the
Smith--Pedicchio commutator $[\nabla_A, \nabla_A]_{\mathsf{Smith}}$ of the indiscrete
equivalence relation on $A$ with itself~\cite{Smith,Ped}, a congruence
on $A$  whose normalization is the derived subobject $[A,A]$ of $A$.
The quotient of $\Eq(f)$ by this double equivalence relation $C$ turns out to be a $2$-nilpotent object, and this is what eventually makes the first projection $\Eq(f) \to A$ a trivial extension.

In Section~\ref{sec:n} we then extend the result to the general case of $n$-nilpotent objects (for any $n\ge 1$),
where we show that
 $f \colon A \to B$ is
central relative to the Birkhoff subcategory $\Nil_n(\C)$ of
$n$-nilpotent objects precisely when its kernel $K$ satisfies
\[
  \underbrace{[A,\dots,[A,K]\dots]}_{\text{$n+1$ terms}}\;=\; 0
\]
(Theorem~\ref{thm:n-main}). For $n = 1$ one recovers the classical characterization of $\Ab(\C)$-central extensions, and for $n=2$ Theorem~\ref{thm:main}. The additional ingredient needed for general $n$ is a purely commutator-theoretic result (Proposition~\ref{prop:n-nil}): if $A = K \vee B$ with $K$ normal in $A$, $B$ an $(n-1)$-nilpotent subobject, and the $n$-fold iterated commutator of $A$ with $K$ trivial, then $A$ is $n$-nilpotent.

In the last section we drop the assumption of algebraic coherence. In an arbitrary
semi-abelian category $\C$, where the Three Subobjects Lemma need not hold, nilpotency
can be defined by means of the higher-order Higgins commutators \cite{HVdL, DHVdL}: an
object $X$ is two-nilpotent when the ternary Higgins commutator
$[X,X,X]_{\mathsf{Hig}}$ vanishes. The two-nilpotent objects determine a Birkhoff subcategory
$\Nil^{\mathsf{H}}_2(\C)$ of $\C$, which coincides with $\Nil_2(\C)$ when $\C$ is
algebraically coherent. We prove that an extension $f \colon A \to B$ with kernel $K$ is
central relative to $\Nil^{\mathsf{H}}_2(\C)$ if and only if
\[
  [K,A,A]_{\mathsf{Hig}} \;=\; 0
\]
(Theorem~\ref{thm:higgins}). In the algebraically coherent case this gives back
Theorem~\ref{thm:main}, but the result also applies to semi-abelian categories which are
not algebraically coherent, such as the category of skew braces \cite{SV}. A similar
characterization, in terms of the $(n+1)$-ary Higgins commutator, holds for the
extensions which are central relative to the $n$-nilpotent objects
(Theorem~\ref{thm:higgins-n}).

Our results apply to a wide range of categories. These include all categories
of interest in the sense of Orzech \cite{Orzech}, which are semi-abelian \cite{JMT}
and algebraically coherent \cite{CGVdL}. Examples of categories of interest include the
categories of groups, Lie algebras, associative algebras, and Poisson algebras.
The category of cocommutative Hopf algebras over a field is also semi-abelian \cite{GSV} and
algebraically coherent \cite{GG21}. Further derived examples include
the category of crossed modules in $\C$ (equivalently, of internal groupoids in $\C$),
as well as functor categories with values in $\C$,
where $\C$ is any semi-abelian algebraically coherent category \cite{CGVdL}. For an explicit
formula for the commutator in $\Omega$-groups, which include all categories of
interest, see Theorem 4B of \cite{Hig56}. This description is refined for
categories of interest in the proof of Theorem 5.3.6 of \cite{Cig09}
(which later appeared as Theorem 4.4 of \cite{CGVdL-NH}).
 For cocommutative
Hopf algebras see e.g. Section $4$ in  \cite{GSV}.

\section{Preliminaries}

\subsection{Semi-abelian and algebraically coherent categories}

Recall that a category is \emph{regular} when it is finitely complete,
has coequalizers of kernel pairs, and regular epimorphisms are stable
under pullback; it is \emph{Barr-exact} when, moreover, every
equivalence relation is effective, i.e.\ is the kernel pair of its own
coequalizer. A finitely complete category is a \emph{Mal'tsev} category
when every reflexive relation in it is an equivalence relation;
equivalently, in the regular case, when equivalence relations on a common
object permute~\cite{CLP}. A pointed finitely complete category is
\emph{protomodular} in the sense of Bourn~\cite{Bourn91} when the Split Short Five Lemma
holds in it; protomodular categories are always Mal'tsev categories.

\begin{definition}[\cite{JMT}]
A category $\C$ is \emph{semi-abelian} when it is pointed, Barr-exact,
protomodular, and has binary coproducts.
\end{definition}

The varieties of groups, of non-unital rings, of Lie algebras over a
commutative ring, of associative and of Leibniz algebras, of crossed
modules, and more generally every category of interest in the sense of
Orzech \cite{Orzech}, are semi-abelian; so are the category of cocommutative Hopf
algebras over a field \cite{GSV} and the dual of the category of pointed sets \cite{JMT}. In a semi-abelian category one has a
well-behaved notion of exact sequence and a remarkable theory of
commutators which will be of importance for the present paper (see 
Section~\ref{ssec:huq}).

The second standing hypothesis is a condition on the behaviour of
 \emph{binary joins} under change-of-base in the fibration of points. For an object
$Y$ of $\C$, let $\Pt_Y(\C)$ denote the category of points over
$Y$, that is, of split epimorphisms with codomain $Y$ together with a
chosen splitting; any morphism $f \colon X \to Y$ induces a
change-of-base functor $f^{*} \colon \Pt_Y(\C) \to
\Pt_X(\C)$ given by pulling back along $f$.

\begin{definition}[\cite{CGVdL}]
A finitely complete category $\C$ is \emph{algebraically coherent} when,
for every morphism $f$ in $\C$, the change-of-base functor $f^{*}$ is
\emph{coherent}, that is, preserves finite limits and jointly strongly
epimorphic pairs of morphisms.
\end{definition}

Algebraic coherence is a rather mild condition in practice: all the
semi-abelian categories listed above are algebraically coherent. A
notable example of a semi-abelian variety which is \emph{not}
algebraically coherent is the variety of non-associative rings. What
 algebraic coherence guarantees
is a very good behaviour of commutators: in an algebraically
coherent semi-abelian category the Huq commutator distributes over joins of
arbitrary subobjects, in the sense that
\[
  [X \vee Y, Z] \;=\; [X,Z] \vee [Y,Z]
\]
for subobjects $X$, $Y$, $Z$ of a common object  
$A$, and the Three
Subobjects Lemma recalled in the next subsection holds \cite{CGVdL}.
The fact that algebraic coherence implies distributivity of the Huq commutator
seems to be known, but we couldn't find an explicit reference and have therefore 
included a proof, see Proposition \ref{alg_coh=>Huq_dist} below.
Neither of these properties is available in
an arbitrary semi-abelian category.

\subsection{Huq centrality and the Huq commutator}\label{ssec:huq}

Let $\C$ be a pointed category with finite limits.

\begin{definition}[\cite{Huq}]
Two morphisms $x \colon X \to A$ and $y \colon Y \to A$ are said to
\emph{cooperate} (or to \emph{commute in the sense of Huq}) when there exists a morphism
$\varphi \colon X \times Y \to A$, called a \emph{cooperator} for $x$ and
$y$, such that
\[
  \varphi \circ \langle 1_X, 0 \rangle = x
  \qquad\text{and}\qquad
  \varphi \circ \langle 0, 1_Y \rangle = y .
\]
\end{definition}

Thus $x$ and $y$ cooperate exactly when the outer part of the diagram
\[
\begin{tikzcd}[ampersand replacement=\&, column sep=large, row sep=large]
X \arrow[r, "{\langle 1_X,\, 0 \rangle}"] \arrow[rd, "{x}"']
\& X \times Y \arrow[d, dashed, "{\varphi}"]
\& Y \arrow[l, "{\langle 0,\, 1_Y \rangle}"'] \arrow[ld, "{y}"] \\
\& A \&
\end{tikzcd}
\]
can be filled in by a (dashed) morphism $\varphi$ making both triangles
commute. In a semi-abelian category (more generally, in any unital category \cite{BB}) such a $\varphi$ is unique when it exists, so that cooperating
is a property of the pair $(x,y)$ and not extra structure. For groups the existence of a cooperator says precisely that every element of the image
of $x$ commutes with every element of the image of $y$.

\begin{definition}
Let $x \colon X \to A$ and $y \colon Y \to A$ be subobjects of an object
 $A$ of a semi-abelian category. The
\emph{Huq commutator} $[X,Y]$ is the smallest normal subobject $N$ of
$A$ with the property that the composites $q \circ x$ and $q \circ y$
cooperate, where $q \colon A \to A/N$ denotes the quotient of $A$ by $N$.
We say that $X$ and $Y$ commute when $[X,Y] = 0$.
\end{definition}

We shall use freely the following standard facts, valid in any
semi-abelian category. The Huq commutator is symmetric and monotone in
each variable. If $X$ and $Y$ are normal subobjects of $A$, then
 \[
  [X,Y] \;\leq\; X \wedge Y. \]
Finally, the Huq commutator is
 preserved by regular images, in the sense that
\[
  g\bigl([X,Y]\bigr) \;=\; \bigl[g(X),g(Y)\bigr]
\]
for every regular epimorphism $g$ with domain $A$, where $g(-)$ denotes
the direct image along $g$. For proofs of these properties we refer to \cite{Huq} and to
the monograph~\cite[Chapter~4]{BB}.

An object $A$ is \emph{abelian} exactly when $[A,A] = 0$, that is,
when $1_A$ cooperates with itself; and a normal subobject $K$ of $A$ is
\emph{central} exactly when $[K,A] = 0$.
Note that if a subobject $x \colon X \to A$ is abelian as an object, then $[X,X] = 0$ in $A$ as well.

By an \emph{extension} we mean a regular epimorphism $f \colon A \to B$
in $\C$, and we write $K = \Ker(f)$ for its kernel, viewed as a normal
subobject of $A$. The classical condition
characterizing the extensions which are central relative to the Birkhoff subcategory $\Ab(\C)$ (in the sense recalled in Section~\ref{sec:galois} below) is precisely that $[K, A]=0$ \cite{Gran-Applications}.

Iterated Huq commutators are governed by the following categorical
version of the classical Three Subgroups Lemma of P.~Hall~\cite{Hall34},
which holds in the algebraically coherent setting~\cite{CGVdL}.

\begin{lemma}[Three Subobjects Lemma]\label{lem:three}
Let $\C$ be an algebraically coherent semi-abelian category, and let
$X$, $Y$ and $Z$ be normal subobjects of an object $A$ of $\C$. Then
\[
  [[X,Y],Z] \;\leq\; [[Y,Z],X] \vee [[Z,X],Y] .
\]
\end{lemma}

This is a highly non-trivial statement in the categorical context: it
is one of the properties of the Huq commutator which fail in an arbitrary
semi-abelian category, and whose validity under algebraic coherence
justifies our standing hypotheses. For instance, the Three Subobjects Lemma was recently shown to fail in the semi-abelian categories of skew braces and of Hopf braces \cite{SV}. Note that the three commutators appearing in the statement are normal subobjects of $A$, since Huq commutators are normal by definition.

We record the instance that will be used below. Let $f \colon A \to B$ be
an extension with kernel $K$, and apply Lemma~\ref{lem:three} with
$X = Y = A$ and $Z = K$, which is legitimate since $A$ and $K$ are both
normal subobjects of $A$. This gives
\[
  [[A,A],K] \;\leq\; [[A,K],A] \vee [[K,A],A] .
\]
Since the Huq commutator is symmetric, both terms on the right hand side
 are equal to $[A,[K,A]]$, and one obtains
\begin{equation}\label{eq:three-instance}
  [K,[A,A]] \;=\; [[A,A],K] \;\leq\; [A,[K,A]] .
\end{equation}
In particular the centrality condition $[A,[K,A]] = 0$ of
Theorem~\ref{thm:main} implies $[K,[A,A]] = 0$.

\begin{remark}\label{rem:strict}
The  inequality~\eqref{eq:three-instance} may be strict, even in the
category $\Grp$ of groups: for the extension $f \colon A = S_3 \to C_2$ given by the
sign, with kernel $K = C_3 = [S_3,S_3]$, one has $[K,[A,A]] = [C_3,C_3] = 1$ while
$[A,[K,A]] = [S_3,C_3] = C_3$. By Theorem~\ref{thm:main} below, this extension is therefore not central relative to $\Nil_2(\Grp)$, although $[K,[A,A]]=1$. \end{remark}

Recall that in a semi-abelian category the join of subobjects $m:S\to A$ and $n:T\to A$
can be constructed as \emph{mono part} of the (regular epi,mono)-decomposition of the induced
morphism $[m,n]:S+T\to A$, where $S+T$ is the coproduct of $S$ and $T$.
Note that by Theorem 6.27 of \cite{CGVdL} algebraic coherence implies that the change-of-base functors
of the fibration of points along split epimorphisms preserve finite colimits.
\begin{proposition}\label{alg_coh=>Huq_dist}
Let $\C$ be a semi-abelian category.
If $\C$ is algebraically coherent (or, more generally, if the change-of-base functors of the fibration of points along the split epimorphisms of the form $B\to 0$
preserve binary coproducts), then Huq commutators distribute over joins: for any subobjects $X,Y$ and $Z$ of an object $A$ in $\C$,
\[[X\vee Y, Z]  = [X,Z]\vee[Y,Z].\]
\end{proposition}
\begin{proof}
Suppose $f:A\to C$ and $f':A'\to C$ are morphisms in $\C$ which both Huq-commute with
a morphism $g:B\to C$ in $\C$.
We obtain the diagram of morphisms of split extensions
\[
\xymatrix@!@C=-7ex@R=-7ex{
&
&
A
\ar[rrr]^{\langle 1,0\rangle}
\ar[ld]_{\iota_1}
\ar[ddd]_{f}
&
&
&
A\times B
\ar@<0.5ex>[rrr]^{\pi_2}
\ar[ld]_{\iota_1\times 1}
\ar[ddd]^{\langle \varphi, \pi_2\rangle}
&
&
&
B
\ar@<0.5ex>[lll]^{\langle 0,1\rangle}
\ar@{=}[ld]
\ar@{=}[ddd]
\\
&
A+A'
\ar[rrr]^{\langle 1,0\rangle}
\ar[ddd]_{[f,f']}
&
&
&
(A+A')\times B
\ar@<0.5ex>[rrr]^{\pi_2}
\ar@{-->}[ddd]^{\theta}
&
&
&
B
\ar@<0.5ex>[lll]^{\langle 0,1\rangle}
\ar@{=}[ddd]
\\
A'
\ar[rrr]^{\langle 1,0\rangle}
\ar[ddd]_{f'}
\ar[ru]^{\iota_2}
&
&
&
A'\times B
\ar@<0.5ex>[rrr]^{\pi_2}
\ar[ddd]^{\langle \varphi', \pi_2\rangle}
\ar[ru]^{\iota_2\times 1}
&
&
&
B
\ar@<0.5ex>[lll]^{\langle 0,1\rangle}
\ar@{=}[ddd]
\ar@{=}[ru]
\\
&
&
C
\ar[rrr]^{\langle 1,0\rangle}
\ar@{=}[ld]
&
&
&
C\times B
\ar@<0.5ex>[rrr]^{\pi_2}
\ar@{=}[ld]
&
&
&
B
\ar@<0.5ex>[lll]^{\langle g,1\rangle}
\ar@{=}[ld]
\\
&
C
\ar[rrr]^{\langle 1,0\rangle}
&
&
&
C\times B
\ar@<0.5ex>[rrr]^{\pi_2}
&
&
&
B
\ar@<0.5ex>[lll]^{\langle g,1\rangle}
\\
C
\ar[rrr]^{\langle 1,0\rangle}
\ar@{=}[ru]
&
&
&
C\times B
\ar@<0.5ex>[rrr]^{\pi_2}
\ar@{=}[ru]
&
&
&
B
\ar@<0.5ex>[lll]^{\langle g,1\rangle}
\ar@{=}[ru]
}
\]
where $\varphi$ and $\varphi'$ are the respective cooperators showing that $f$ and $g$, and $f'$ and $g$ Huq-commute.
Since $(B\to 0)^*: \Pt_{0}(\C) \cong \C \to \Pt_{B}(\C)$ preserves binary coproducts the top part of the diagram is a coproduct in $\Pt_{B}(\C)$
and we obtain the morphism $\theta$ forming part of the induced morphism from the coproduct.
It now follows that $[f,f']:A+A'\to C$ and $g$ Huq-commute with cooperator $\pi_1 \circ  \theta$. Letting $[f,f'] = m\circ e$ be the (regular epi, mono)-factorization of $[f,f']$, with $m \colon f(A)\vee f'(A') \to C$, Proposition 1.6.4 \cite{BB}
tells us the $m$ and $g$ commute.  To keep the paper self contained let us recall why: the cooperator $ \pi_1\circ\theta \colon (A+A')\times B \to C$ of $[f,f']$ and $g$ coequalizes the kernel pair $(\Eq(e)\times B,\, p_1 \times 1_B,\, p_2\times 1_B)$ of the regular epimorphism $e\times 1_B$, where $(\Eq(e), p_1, p_2)$ is the kernel pair of $e$,
(for this one uses the fact that the pair $\langle 1,0\rangle \colon \Eq(e) \to \Eq(e)\times B$ and $\langle 0,1\rangle \colon B \to \Eq(e)\times B$ is jointly epimorphic),
 hence $\pi_1\circ\theta$ factors as $\varphi''\circ(e\times 1_B)$, and $\varphi''$ is a cooperator for $m$ and $g$.
The proof of Theorem 1.1 of \cite{GRAY:2013a} can now be used to complete the proof.  We recall the argument: let $X$, $Y$ and $Z$ be subobjects of an object $A$ of $\C$, and put
$N = [X,Z]\vee[Y,Z]$, with quotient $q \colon A \to A/N$. Since $[X,Z] \leq N$, the
subobjects $q(X)$ and $q(Z)$ of $A/N$ commute, and likewise $q(Y)$ and $q(Z)$. By what
we have just proved, $q(X)\vee q(Y) = q(X \vee Y)$ commutes with $q(Z)$, so that
$[X\vee Y, Z] \leq N = [X,Z]\vee[Y,Z]$. The reverse inequality follows from the
monotonicity of the Huq commutator, and the proposition is proved.
\end{proof}

\subsection{Smith commutators and the condition (SH)}\label{ssec:smith}
For an object $A$ of $\C$ we write $\Delta_A$ for the discrete (smallest)
equivalence relation on $A$ and $\nabla_A$ for the indiscrete (largest) one; for an
extension $f \colon A \to B$ we write $\Eq(f)$ for its kernel pair,
regarded as an (effective) equivalence relation on $A$.

Every equivalence relation gives rise to a normal subobject of the
object it lives on, and this passage is what allows the Smith commutator
and the Huq commutator to be compared \cite{BG02}.

\begin{definition}
Let $(R, r_1, r_2)$ be an equivalence relation on an object $A$. Its
\emph{normalization} $I_R$ is the normal subobject of $A$ obtained as the
composite
\[
\begin{tikzcd}[ampersand replacement=\&, column sep=large]
I_R \arrow[r, tail, "{\ker(r_1)}"] \& R \arrow[r, "{r_2}"] \& A ,
\end{tikzcd}
\]
that is, the image under the second projection of the kernel of the
first. Since $R$ is a relation, the composite is a monomorphism, and one
shows that $I_R$ is normal in $A$ \cite{Bourn91, BB}.
\end{definition}

Informally, $I_R$ collects those elements of $A$ which are
$R$-equivalent to $0$.  For the kernel pair of an extension
$f \colon A \to B$ with kernel $K$ one has $I_{\Eq(f)} = K$, and $I_{\nabla_A} = A$. In a
semi-abelian category the assignment $R \mapsto I_R$ is an isomorphism
between the lattice of equivalence relations  on $A$ and the lattice of
normal subobjects of $A$~\cite{Bourn91, BG02}.

Given two equivalence relations $R$ and $S$ on $A$, we write $[R,S]_{\mathsf{Smith}}$ for
the Smith--Pedicchio commutator of $R$ and $S$~\cite{Smith,Ped}. One has
$[R,S]_{\mathsf{Smith}} = \Delta_A$ precisely when $R$ and $S$ \emph{centralize} each
other, which in the present exact Mal'tsev context is equivalent to the
existence of a (then necessarily unique) \emph{double centralizing
equivalence relation} $C$ on $R$ and $S$: a double equivalence relation
\[
\begin{tikzcd}[ampersand replacement=\&, column sep=huge, row sep=large]
C
  \arrow[r, shift left=0.9ex, "{\pi_1}"] \arrow[r, shift right=0.9ex, "{\pi_2}"']
  \arrow[d, shift left=0.9ex, "{\rho_1}"] \arrow[d, shift right=0.9ex, "{\rho_2}"']
\& S
  \arrow[d, shift left=0.9ex, "{s_1}"] \arrow[d, shift right=0.9ex, "{s_2}"'] \\
R
  \arrow[r, shift left=0.9ex, "{r_1}"] \arrow[r, shift right=0.9ex, "{r_2}"']
\& A
\end{tikzcd}
\]
in which $r_1, r_2$ and $s_1, s_2$ are the projections of $R$ and of $S$,
while $\pi_1, \pi_2$ and $\rho_1, \rho_2$ are the two pairs of
projections of $C$, seen as an equivalence relation on $S$ and on $R$
respectively, so that $r_i \circ \rho_j = s_j \circ \pi_i$ for
$i,j \in \{1,2\}$. That $C$ is \emph{centralizing} means that the induced
comparison $\langle \rho_1, \pi_1 \rangle \colon C \to R \times_A S$ (the
pullback of $r_1$ and $s_1$)  is an isomorphism; in that
case each of the four commutative squares $r_i \circ \rho_j = s_j \circ
\pi_i$ is a pullback. 

 The category $\C$ satisfies the \emph{Smith is Huq} condition, denoted
by (SH), when for any pair of equivalence relations $R$ and $S$ on the
same object $A$ one has
\[
  [R,S]_{\mathsf{Smith}} = \Delta_A
  \qquad \Longleftrightarrow \qquad
  [I_R, I_S] = 0 ,
\]
the commutator on the right being the Huq commutator of the
normalizations. The implication from left to right holds in any
semi-abelian category, and so does the equivalence when one of the two
relations is $\nabla_A$ \cite{GVdL}; the content of (SH) is the
implication from right to left for arbitrary $R$ and $S$. The condition
(SH) holds in every algebraically coherent semi-abelian category, as
shown in~\cite[Section 6]{CGVdL}.

The normalization of a Smith commutator of equivalence
relations $R$ and $S$ is also called the Ursini commutator of $I_R$ and $I_S$ in \cite{Mantovani}. Under the (SH) condition it coincides with the Huq commutator of the normalizations, that is,
$I_{[R,S]_{\mathsf{Smith}}} = [I_R,I_S]$ \cite{MFVdL, Mantovani}.

In particular, for an extension $f \colon A \to B$
with kernel $K$, the commutator $[\Eq(f),\nabla_A]_{\mathsf{Smith}}$ has
normalization $[K,A]$, and $[\nabla_A,\nabla_A]_{\mathsf{Smith}}$ has normalization
$[A,A]$.

Finally, recall that an internal reflexive graph $\mathbb{G} $ in $\C$ is a diagram
$G_1 \rightrightarrows G_0$ together with a common splitting of the two
morphisms, and that in a Mal'tsev category being an internal groupoid is
a property of a reflexive graph rather than extra structure \cite{CLP}. As shown
in~\cite{Gran99}, in a regular Mal'tsev category internal groupoids are
closed in the category of reflexive graphs under regular quotients: if
$\mathbb{G} \to \mathbb{H}$ is a morphism of reflexive graphs which is
levelwise a regular epimorphism and $\mathbb{G}$ is an internal groupoid,
then so is $\mathbb{H}$.

\section{Central extensions in the sense of Galois theory}\label{sec:galois}
From now on $\C$ denotes a semi-abelian category.
 \begin{definition}
An object $A$ of $\C$ is \emph{$2$-nilpotent} (nilpotent of class at most
$2$) when $[[A,A],A] = 0$. We write $\Nil_2(\C)$ for the full subcategory
of $\C$ determined by the $2$-nilpotent objects.
\end{definition}

The subcategory $\Nil_2(\C)$ is a Birkhoff subcategory of $\C$ \cite{Huq, JK94}: it is
full, closed under subobjects and regular quotients (this essentially follows from the good properties of the Huq commutator recalled in Section~\ref{ssec:huq}), and the reflector
\[
  \operatorname{nil}_2 \colon \C \longrightarrow \Nil_2(\C),
  \qquad
  A \longmapsto A/[[A,A],A],
\]
is left adjoint to the inclusion $\Nil_2(\C) \rightarrow \C$, the unit at $A$ being the canonical
quotient $\eta_A \colon A \to A/[[A,A],A]$. For $\C = \Grp$ one recovers
the classical variety of groups of nilpotency class at most~$2$. There
are Birkhoff inclusions
\[
  \Ab(\C) \;\subseteq\; \Nil_2(\C) \;\subseteq\; \C ,
\]
and the characterization below is the second step of the resulting
nilpotency tower, the first step being the description of
$\Ab(\C)$-central extensions recalled in the preliminaries.

We recall the relevant notions from~\cite{JK94,JK00}. Let
$I \colon \C \to \mathcal{X}$ be the reflector onto a Birkhoff
subcategory $\mathcal{X}$ of $\C$, with unit $\eta$.

\begin{definition}\label{def:trivial}
An extension $f \colon A \to B$ is \emph{trivial} (relative to
$\mathcal{X}$) when the naturality square
\[
\begin{tikzcd}[ampersand replacement=\&, column sep=large, row sep=large]
A \arrow[r, "{\eta_A}"] \arrow[d, "{f}"'] \& I(A) \arrow[d, "{I(f)}"] \\
B \arrow[r, "{\eta_B}"'] \& I(B)
\end{tikzcd}
\]
is a pullback. It is \emph{central} when there exists an extension
$p \colon E \to B$ such that the pullback $\pi_1 \colon E \times_B A \to
E$ of $f$ along $p$ is a trivial extension, and it is \emph{normal} when
one may take $p = f$, that is, when the first projection
$f_1 \colon \Eq(f) \to A$ of its kernel pair  is a trivial extension.
\end{definition}
 Every normal extension is central by
definition, and a fundamental result of Janelidze and Kelly~\cite{JK94}
asserts that the converse holds when $\C$ is Barr-exact Mal'tsev and
$\mathcal{X}$ is a Birkhoff subcategory, which is the case here. Recall also that, since $\mathcal{X}$ is a Birkhoff subcategory, the naturality square above is always a \emph{regular pushout} \cite{JK94, Bourn03} (also called a \emph{double extension} \cite{Jan91}): this means that all its arrows are regular epimorphisms, and so is the comparison morphism $\langle f, \eta_A\rangle \colon A \to B \times_{I(B)} I(A)$ to the pullback of $I(f)$ along $\eta_B$. Consequently, the square is a pullback as soon as $f$ and $\eta_A$ are jointly monomorphic, since the comparison morphism is then both a monomorphism and a regular epimorphism. The following pullback stability
result is  the first step of our argument.

\begin{lemma}\label{lem:pullback}
Let $f \colon A \to B$ and $p \colon E \to B$ be extensions in $\C$, and
consider the pullback of $f$ along $p$:
\[
\begin{tikzcd}[ampersand replacement=\&, column sep=large, row sep=large]
E \times_B A \arrow[r, "{\pi_2}"] \arrow[d, "{\pi_1}"'] \& A \arrow[d, "{f}"] \\
E \arrow[r, "{p}"'] \& B .
\end{tikzcd}
\]
Write $K = \Ker(f)$ and identify, as usual, $\Ker(\pi_1)$ with $K$.
Then
\[
  [A,[K,A]] = 0
  \qquad \Longleftrightarrow \qquad
  \bigl[E \times_B A,\; [K, E \times_B A]\bigr] = 0 .
\]
\end{lemma}

\begin{proof}
Write $P = E \times_B A$ and $L = \Ker(\pi_2)$, so that $L$ is
canonically isomorphic to $\Ker(p)$. Since $p$ is a regular epimorphism,
so is its pullback $\pi_2$; and the restriction of $\pi_2$ to
$\Ker(\pi_1)$ is precisely the inclusion $K \to A$, whence
$\pi_2(K) = K$. Moreover $K \wedge L = 0$,  since
$K \wedge L = \Ker(\pi_1) \wedge \Ker(\pi_2)$, with $\pi_1$ and $\pi_2$ jointly monomorphic.

Suppose first that $[P,[K,P]] = 0$. Applying the regular epimorphism
$\pi_2$, and using that regular epimorphisms preserve Huq commutators,
we get
\[
  0 \;=\; \pi_2\bigl([P,[K,P]]\bigr)
    \;=\; \bigl[\pi_2(P), [\pi_2(K),\pi_2(P)]\bigr]
    \;=\; [A,[K,A]] .
\]

Conversely, suppose that $[A,[K,A]] = 0$, and put $M = [K,P]$, a normal
subobject of $P$ with $M \leq K \wedge P = K$. On the one hand
$[P,M] \leq M \leq K$, the Huq commutator of two normal subobjects being
contained in their meet. On the other hand $\pi_2(M) = [\pi_2(K),
\pi_2(P)] = [K,A]$, whence
\[
  \pi_2\bigl([P,M]\bigr) \;=\; [\pi_2(P), \pi_2(M)] \;=\; [A,[K,A]] \;=\; 0 ,
\]
so that $[P,M] \leq \Ker(\pi_2) = L$. Therefore
$[P,[K,P]] = [P,M] \leq K \wedge L = 0$, as desired.
\end{proof}

\begin{remark}
The proof above uses no hypothesis beyond semi-abelianness; in
particular, neither algebraic coherence nor the condition (SH) is needed
for Lemma~\ref{lem:pullback}. We'll see that Lemma~\ref{lem:pullback} is the case $n=2$ of Lemma~\ref{lem:pullback-n} below, whose proof is essentially the same.
\end{remark}

\section{The characterization}

\begin{theorem}\label{thm:main}
Let $\C$ be a semi-abelian, algebraically coherent category, and let
$f \colon A \to B$ be an extension in $\C$ with kernel $K$.
Then $f$ is a central extension relative to the Birkhoff subcategory
$\Nil_2(\C)$ of $2$-nilpotent objects if and only if
\[
  [A,[K,A]] \;=\; 0 ,
\]
where $[\,\cdot\,,\cdot\,]$ denotes the Huq commutator; equivalently, if
and only if the normal subobject $[K,A]$ of $A$ is central in $A$.
\end{theorem}

\begin{proof}
Assume that $[A,[K,A]] = 0$. By the instance~\eqref{eq:three-instance} of
the Three Subobjects Lemma this implies
\[
  [K,[A,A]] \;=\; 0 .
\]
Now the normalization of the equivalence relation $\Eq(f)$ on $A$ is $K$,
and the normalization of $[\nabla_A,\nabla_A]_{\mathsf{Smith}}$ is $[A,A]$; the condition
(SH), which holds since $\C$ is algebraically coherent, therefore turns
the vanishing of the Huq commutator $[K,[A,A]]$ into
the vanishing of the Smith commutator
\[
  \bigl[\,\Eq(f),\; [\nabla_A,\nabla_A]_{\mathsf{Smith}} \,\bigr]_{\mathsf{Smith}} \;=\; \Delta_A .
\]
Accordingly, the equivalence relations $\Eq(f)$ and $[\nabla_A,\nabla_A]_{\mathsf{Smith}}$
centralize each other, and there is a unique double centralizing
equivalence relation $C$ on them. Taking the quotients in the vertical
direction, we obtain the diagram
\begin{equation}\label{diag:key}
\begin{tikzcd}[ampersand replacement=\&, column sep=huge, row sep=large]
C
  \arrow[r, shift left=0.9ex, "{p_1}"] \arrow[r, shift right=0.9ex, "{p_2}"']
  \arrow[d, shift left=0.9ex, "{c_1}"] \arrow[d, shift right=0.9ex, "{c_2}"']
\& {[\nabla_A,\nabla_A]_{\mathsf{Smith}}}
  \arrow[d, shift left=0.9ex, "{d_1}"] \arrow[d, shift right=0.9ex, "{d_2}"'] \\
\Eq(f)
  \arrow[r, shift left=0.9ex, "{f_1}"] \arrow[r, shift right=0.9ex, "{f_2}"']
  \arrow[d, two heads, "{\psi}"']
\& A
  \arrow[d, two heads, "{\varphi}"] \\
\Eq(f)/C
  \arrow[r, shift left=0.9ex, "{\bar{f}_1}"]
  \arrow[r, shift right=0.9ex, "{\bar{f}_2}"']
\& A/[\nabla_A,\nabla_A]_{\mathsf{Smith}}
\end{tikzcd}
\end{equation}
in which $f_1, f_2$ are the projections of the kernel pair of $f$, the
pair $d_1, d_2$ consists of the projections of the equivalence relation
$[\nabla_A,\nabla_A]_{\mathsf{Smith}}$ on $A$, and $c_1, c_2$ and $p_1, p_2$ are the two
pairs of projections of the double equivalence relation $C$, the first
pair over $\Eq(f)$ and the second over $[\nabla_A,\nabla_A]_{\mathsf{Smith}}$, so that
$f_i \circ c_j = d_j \circ p_i$ for $i,j \in \{1,2\}$. The
morphisms $\psi$ and $\varphi$ are the quotients of the two vertical
equivalence relations, the latter having codomain
$A/[\nabla_A,\nabla_A]_{\mathsf{Smith}} \cong A/[A,A]$. The
composite $\varphi \circ f_i$ coequalizes $c_1$ and $c_2$, and there are
unique induced morphisms $\bar{f}_1, \bar{f}_2$ with
 $\bar{f}_i \circ \psi = \varphi \circ f_i$ for $i = 1,2$. To lighten the
notation we shall write
\[
  E \;=\; \Eq(f)/C \qquad\text{and}\qquad Q \;=\; A/[\nabla_A,\nabla_A]_{\mathsf{Smith}}
\]
for the quotients appearing in the lower left and right corner
of~\eqref{diag:key}, respectively.

The lower part of~\eqref{diag:key} is then a morphism of internal
reflexive graphs
\[
  \bigl(\Eq(f) \rightrightarrows A\bigr)
  \longrightarrow
  \bigl(E \rightrightarrows Q\bigr) ,
\]
levelwise a regular epimorphism.  The codomain of this morphism is again an internal
groupoid: the domain $\Eq(f) \rightrightarrows A$ is one, being an
equivalence relation, and internal groupoids are closed under regular
quotients in any regular Mal'tsev category~\cite{Gran99}, hence in
particular in the semi-abelian category $\C$.
 We next observe that this morphism is a \emph{discrete fibration},
i.e.\ that the square
  \begin{equation}\label{diag:pb}
\begin{tikzcd}[ampersand replacement=\&, column sep=huge, row sep=large]
\Eq(f) \arrow[r, two heads, "{\psi}"] \arrow[d, "{f_1}"'] \& E
  \arrow[d, "{\bar{f}_1}"] \\
A \arrow[r, two heads, "{\varphi}"'] \& Q
\end{tikzcd}
\end{equation}
is a pullback. Indeed, the two vertical columns of~\eqref{diag:key} are the kernel pairs of $\psi$ and of $\varphi$, respectively, and the upper squares are pullbacks because $C$ is a centralizing double relation. Since $\varphi$ is a regular epimorphism, hence an effective descent morphism in the exact category $\C$, the Barr--Kock Theorem (see Example 6.10 in \cite{Barr}) implies that each of the lower squares is a pullback as well.
The next step in our proof is to show that $E$ is $2$-nilpotent. Since $\Eq(f)$ is
obtained by pulling back $f$ along itself, Lemma \ref{lem:pullback} tells us that
$[\Eq(f),[K,\Eq(f)]]=0$. Since \eqref{diag:pb} is a pullback, the kernel of $\bar{f}_1$ is (isomorphic to) $K$ and $\psi(K)=K$; as $\psi$ is a regular epimorphism, it follows that $[E,[K,E]]=\psi([\Eq(f),[K,\Eq(f)]]) = 0$. Since $\bar{f}_1$ is a split epimorphism (the graph $E \rightrightarrows Q$ being reflexive),
we can view $Q$ as a subobject of $E$ and, as for any split extension in
a semi-abelian category, $E=K\vee Q$. Moreover, $Q \cong A/[A,A]$ is abelian, so that $[Q,Q]=0$ also when the commutator is computed in $E$ (see Section~\ref{ssec:huq}). Using
the fact that Huq commutators distribute over joins and are symmetric and monotone,
together with $[E,[K,E]]=0$ and $[Q,Q]=0$, we obtain
\begin{align*}
[E,[E,E]] &= [E,[K\vee Q,K\vee Q]]\\
&= [E,[K,K]] \vee [E,[K,Q]] \vee [E,[Q,K]] \vee [E,[Q,Q]]\\
&\leq [E,[K,E]] \vee [E,[K,E]] \vee [E,[K,E]] \vee [E,0]\\
&=0,
\end{align*}
as desired.
Let now $I = \operatorname{nil}_2 \colon \C \to \Nil_2(\C)$ be the reflector, with unit $\eta$. Since $E$ and $Q$ are $2$-nilpotent, via the universal property of the reflector there are unique morphisms $u$ and $v$ making the diagram
\[
\begin{tikzcd}[ampersand replacement=\&, column sep=huge, row sep=large]
\Eq(f) \arrow[bend left=20,rr, two heads, "{\psi}"] \arrow[r, two heads, "{\eta_{\Eq(f)}}"] \arrow[d, "{f_1}"'] \& I(\Eq(f)) \arrow[r, two heads, "{u}"] \arrow[d, "{I(f_1)}"'] \& E
  \arrow[d, "{\bar{f}_1}"] \\
A \arrow[bend right=20,rr, two heads, "{\varphi}"']\arrow[r, two heads, "{\eta_{A}}"'] \& I(A)\arrow[r, two heads, "{v}"'] \& Q
\end{tikzcd}
\]
commute. The square~\eqref{diag:pb} being a pullback, the morphisms $\psi$ and $f_1$ are jointly monomorphic, hence so are $\eta_{\Eq(f)}$ and $f_1$. Since the left hand square is a regular pushout, as recalled after Definition~\ref{def:trivial}, it follows that it is a pullback (see also Lemma 1.1 in \cite{Gran01}), so that $f_1$
 is a trivial extension and $f$ is a normal, hence central, extension.

Conversely, suppose that $f$ is central. By the results in \cite{JK94} it is then
normal, i.e.\ $f_1 \colon \Eq(f) \to A$ is a trivial extension, so that
$\eta_{\Eq(f)}$ and $f_1$ are jointly monomorphic and, consequently,
$\Ker(\eta_{\Eq(f)}) \wedge K = 0$, where
$$\Ker(\eta_{\Eq(f)}) = [\Eq(f),[\Eq(f),\Eq(f)]]$$ and $K = \Ker(f_1)$. Since
$[\Eq(f),[K,\Eq(f)]] \leq [\Eq(f),[\Eq(f),\Eq(f)]] \wedge K$ by monotonicity
and normality, it follows that $[\Eq(f),[K,\Eq(f)]] = 0$, and
Lemma~\ref{lem:pullback}, applied to the pullback of $f$ along itself,
gives $[A,[K,A]] = 0$.
\end{proof}
\begin{remark}\label{rem:decomp}
The proof above shows that, similarly to the case 
of central extensions relative to $\Ab(\C)$ \cite{Gran-Applications}, the kernel pair of a $\Nil_2(\C)$-central extension $f \colon A \to B$ decomposes as a pullback: $\Eq(f) \cong A \times_{Q} E$, where $Q = A/[A,A]$ is abelian and $E = \Eq(f)/C$ is $2$-nilpotent.
\end{remark}
\section{The $n$-nilpotent case}\label{sec:n}

\newcommand{\gam}[3][]{{}^{\vphantom{#1}}_{#2}\!\gamma^{#1}_{#3}}
\newcommand{\nabk}[1]{\nabla_A^{[#1]}}

In this section we show that our characterization generalises to the Birkhoff subcategory of $n$-nilpotent objects, for any $n \geq 1$. Throughout the section, $\C$ is a semi-abelian algebraically coherent category.
We begin by fixing some notation for iterated commutators.
For an object $A$ and a subobject $S$ of $A$, let us denote by $\gam{A}{S}$
the map assigning to each subobject $T$ of $A$ the Huq commutator $[S,T]$ of
$S$ and $T$, computed in $A$. The subscript $A$ on the left records the object in which the commutator is computed: for $n \geq 1$, 
\[
  \gam[n]{A}{S}(T) \;=\; \underbrace{[S,[S,\dots,[S,T]\dots]]}_{\text{$n$ occurrences of $S$}} .
\]
Each $\gam[n]{A}{S}$ is monotone, and it distributes over joins of subobjects since the Huq commutator does. When $S$ and $T$ are normal in $A$, so is $\gam[n]{A}{S}(T)$. In particular, $\gam[n]{A}{A}(A)$ is the $n$-th term of the lower central series of $A$ (indexed so that its $0$-th term is $A$ itself), and $\gam[n]{A}{A}(K)$ is the subobject appearing in the statement of Theorem~\ref{thm:n-main} below.

\begin{definition}
An object $A$ of $\C$ is \emph{$n$-nilpotent} (nilpotent of class at most $n$) when $\gam[n]{A}{A}(A) = 0$. We write $\Nil_n(\C)$ for the full subcategory of $\C$ determined by the $n$-nilpotent objects.
\end{definition}
The definition makes sense for $n = 0$ as well: since $\gam[0]{A}{A}$ is the identity, an object $A$ is $0$-nilpotent precisely when $A = 0$, so that $\Nil_0(\C)$ is the full subcategory of $\C$ consisting of the zero objects.
Thus $\Nil_1(\C) = \Ab(\C)$, and for $n = 2$ we recover the definition of Section~\ref{sec:galois}. As in that case, $\Nil_n(\C)$ is a Birkhoff subcategory of $\C$ \cite{Huq} (see also \cite{Gray19} for nilpotency in semi-abelian categories), the reflector being $A \mapsto A/\gam[n]{A}{A}(A)$ with the canonical quotient as unit. Note that $\gam[n]{A}{A}(A) \leq \gam[n-1]{A}{A}(A)$, so that an $(n-1)$-nilpotent object is $n$-nilpotent, and there is a tower of Birkhoff subcategories $\Ab(\C) = \Nil_1(\C) \subseteq \Nil_2(\C) \subseteq \cdots \subseteq \C$.

The aim of this section is to prove the following theorem:
\begin{theorem}\label{thm:n-main}
Let $\C$ be a semi-abelian, algebraically coherent category, let $n \geq 1$, and let
$f \colon A \to B$ be an extension in $\C$ with kernel $K$.
Then $f$ is a central extension relative to the Birkhoff subcategory
$\Nil_n(\C)$ of $n$-nilpotent objects if and only if
\[
  \gam[n]{A}{A}(K)\; = \;\underbrace{[A,\dots,[A,K]\dots]}_{\text{$n+1$ terms}} \;=\; 0 ,
\]
equivalently, if
and only if the normal subobject $\gam[n-1]{A}{A}(K)$ of $A$ is central in $A$.
\end{theorem}
For $n = 1$ this is the classical characterization of $\Ab(\C)$-central extensions recalled in Section~\ref{ssec:huq}, and for $n=2$ it is Theorem~\ref{thm:main}.
We defer the proof until after we establish some preliminary results of a commutator-theoretic nature. The first one generalizes the instance~\eqref{eq:three-instance} of the Three Subobjects Lemma.
\begin{lemma}\label{lem:A}
For a normal subobject $K$ of an object $A$ in a semi-abelian, algebraically coherent category,
and for $n$ a non-negative integer,
\[
[K,\gam[n]{A}{A}(A)] \leq \gam[n+1]{A}{A}(K).
\]
\end{lemma}
\begin{proof}
We prove the claim by induction on $n$, for all normal subobjects $K$ of $A$. For $n=0$ there is nothing to prove since
\[[K,\gam[0]{A}{A}(A)]= [K,A] = [A,K] = \gam[1]{A}{A}(K).\]  Suppose the claim
holds for some $n$ and for every normal subobject of $A$. Applying the Three Subobjects Lemma~\ref{lem:three} to the normal subobjects $A$, $\gam[n]{A}{A}(A)$ and $K$ of $A$, and then the induction hypothesis to the normal subobjects $K$ and $[K,A]$, we get
\begin{align*}
[K,\gam[n+1]{A}{A}(A)] &= [K,[A,\gam[n]{A}{A}(A)]]\\
&\leq [A,[K,\gam[n]{A}{A}(A)]] \vee [\gam[n]{A}{A}(A),[K,A]]\\
&\leq [A,\gam[n+1]{A}{A}(K)] \vee \gam[n+1]{A}{A}([K,A])\\
&= \gam[n+2]{A}{A}(K).
\end{align*}
This proves the claim for $n+1$.
\end{proof}
\begin{lemma}\label{lem:B}
For subobjects $K$ and $B$ of an object $A$ in a semi-abelian algebraically coherent category,
 with $K$ normal in $A$ and $K\vee B=A$, and for $n$ a positive integer,
\[
\gam[n]{A}{A}(A) \leq \gam[n]{A}{A}(K) \vee [A,\gam[n-1]{A}{B}(B)].
\]
\end{lemma}
\begin{proof}
We prove the claim by induction on $n$. For $n=1$ we have
\[\gam[1]{A}{A}(A)= [A,A] = [A,K\vee B]= [A,K] \vee [A,B] = \gam[1]{A}{A}(K) \vee [A,\gam[0]{A}{B}(B)].\]  Suppose the claim
holds for some $n$. Using the distributivity of the Huq commutator over joins, its monotonicity (which gives $\gam[n-1]{A}{B}(B) \leq \gam[n-1]{A}{A}(A)$, since $B \leq A$) and Lemma~\ref{lem:A}, we obtain
\begin{align*}
\gam[n+1]{A}{A}(A)&= [A,\gam[n]{A}{A}(A)]\\
&\leq [A,\gam[n]{A}{A}(K) \vee [A,\gam[n-1]{A}{B}(B)]]\\
&= [A,\gam[n]{A}{A}(K)] \vee [A,[A,\gam[n-1]{A}{B}(B)]]\\
&= \gam[n+1]{A}{A}(K) \vee [A,[K\vee B,\gam[n-1]{A}{B}(B)]]\\
&= \gam[n+1]{A}{A}(K) \vee [A,[K,\gam[n-1]{A}{B}(B)]]\vee [A,[B,\gam[n-1]{A}{B}(B)]]\\
&\leq \gam[n+1]{A}{A}(K) \vee [A,[K,\gam[n-1]{A}{A}(A)]]\vee [A,\gam[n]{A}{B}(B)]\\
&\leq \gam[n+1]{A}{A}(K) \vee [A,\gam[n]{A}{A}(K)]\vee [A,\gam[n]{A}{B}(B)]\\
&= \gam[n+1]{A}{A}(K) \vee [A,\gam[n]{A}{B}(B)].
\end{align*}
This proves the claim for $n+1$.
\end{proof}
Before stating the next lemma, let us recall how Huq commutators behave when the ``ambient object'' changes. Let $X$ be a subobject of an object $Y$ in a semi-abelian category, and let $S$ and $T$ be subobjects of $X$.
Denoting by $[S,T]_X$ and $[S,T]_Y$ the Huq commutators of $S$ and $T$ in $X$ and in $Y$ respectively, it
turns out that $[S,T]_Y$ is the normal closure of $[S,T]_X$ in $Y$, which can
be computed as 
\begin{equation}\label{eq:normal-closure}
[S,T]_Y= [S,T]_X \vee [[S,T]_X,Y]_Y .
\end{equation}
Indeed, on the one hand $[S,T]_X \leq [S,T]_Y$: if $N$ is a normal subobject of $Y$
such that the images of $S$ and $T$ in $Y/N$ cooperate, the cooperator
$S \times T \to Y/N$ factors through the image of $X$ in $Y/N$, because the pair
$\langle 1,0\rangle, \langle 0,1\rangle$ is jointly strongly epimorphic and both
$S$ and $T$ factor through $X$; hence $S$ and $T$ cooperate in $X/(X \wedge N)$,
so that $[S,T]_X \leq X \wedge N$. Applying this to $N = [S,T]_Y$ gives the
inequality, and the same argument run backwards shows that $[S,T]_Y$ is the
smallest normal subobject of $Y$ containing $[S,T]_X$, that is, its normal
closure. On the other hand, by a result of \cite[Proposition~4.14]{HL13} (see also \cite[Proposition~2.18]{HVdL}), the normal closure of a
subobject $Z$ of $Y$ is $Z \vee [Z,Y]_Y$.\footnote{In \cite{HL13, HVdL}
this formula is stated for the Higgins commutator of $Z$ and $Y$ (recalled in Section~\ref{sec:higgins}),
which coincides with the Huq commutator $[Z,Y]_Y$ because one of the two subobjects
is the whole object $Y$ \cite{MM10}.}
In particular, for subobjects $S \leq X \leq Y$ one has $[S,S]_Y = 0$ whenever $[S,S]_X = 0$.
\begin{lemma}\label{lem:C}
Let $K$ and $B$ be subobjects of an object $A$ in a semi-abelian algebraically
coherent category, with $K$ normal in $A$ and $K \vee B = A$, and let $n \geq 1$.
If
\[\gam[n]{A}{A}(K) \;=\; 0 \]
and $B$ is $(n-1)$-nilpotent, then
\[\bigl[A,\, \gam[n-1]{A}{B}(B)\bigr] \;=\; 0. \]
\end{lemma}
\begin{proof}
We first show, by induction on $k \in \{0,\dots,n\}$, that
\begin{equation}\label{eq:lemC-ind}
\gam[k]{A}{B}\bigl(\gam[n-k]{B}{B}(B)\bigr) \;=\; 0 .
\end{equation}
For $k=0$ this is the equality $\gam[n]{B}{B}(B) = 0$, which holds because
$\gam[n]{B}{B}(B) = [B,\gam[n-1]{B}{B}(B)]_B \leq \gam[n-1]{B}{B}(B) = 0$, the
object $B$ being $(n-1)$-nilpotent. Suppose now that \eqref{eq:lemC-ind} holds
for some $k \leq n-1$, and write $S = \gam[n-k-1]{B}{B}(B)$ and
$S' = [B,S]_B = \gam[n-k]{B}{B}(B)$, so that \eqref{eq:lemC-ind} reads
$\gam[k]{A}{B}(S') = 0$. By the description \eqref{eq:normal-closure} of $[B,S]_A$ as the normal closure of
$[B,S]_B$ recalled above, together with $A = K \vee B$ and the distributivity of
the Huq commutator over joins,
\begin{align*}
[B,S]_A &= S' \vee  [A,S']_A
        \;=\; S' \vee  [B,S']_A \vee  [K,S']_A .
\end{align*}
Since $S' = \gam[n-k]{B}{B}(B) \leq \gam[n-k]{A}{A}(A)$ (by \eqref{eq:normal-closure} and monotonicity), Lemma~\ref{lem:A} gives
$[K,S']_A \leq [K,\gam[n-k]{A}{A}(A)] \leq \gam[n-k+1]{A}{A}(K)$, whence
\[
[B,S]_A \;\leq\; S' \vee  [B,S']_A \vee  \gam[n-k+1]{A}{A}(K).
\]
Applying $\gam[k]{A}{B}$, which is monotone and distributes over joins, we obtain
\begin{align*}
\gam[k+1]{A}{B}(S) &= \gam[k]{A}{B}\bigl([B,S]_A\bigr)\\
&\leq \gam[k]{A}{B}(S') \vee \gam[k+1]{A}{B}(S')\vee  \gam[k]{A}{B}\bigl(\gam[n-k+1]{A}{A}(K)\bigr).
\end{align*}
The first term vanishes by the induction hypothesis, and so does the second,
since $\gam[k+1]{A}{B}(S') = [B, \gam[k]{A}{B}(S')]_A = [B,0]_A = 0$. As for the
third, $B \leq A$ and the monotonicity of the commutator give
$\gam[k]{A}{B}(N) \leq \gam[k]{A}{A}(N)$ for every subobject $N$
of $A$; with $N = \gam[n-k+1]{A}{A}(K)$ this yields
$\gam[k]{A}{B}\bigl(\gam[n-k+1]{A}{A}(K)\bigr) \leq
\gam[k]{A}{A}\bigl(\gam[n-k+1]{A}{A}(K)\bigr) = \gam[n+1]{A}{A}(K) \leq
\gam[n]{A}{A}(K) = 0$. Hence $\gam[k+1]{A}{B}(S) = 0$, which is
\eqref{eq:lemC-ind} for $k+1$.

Taking $k = n$ in \eqref{eq:lemC-ind} we get $\gam[n]{A}{B}(B) = 0$, that is,
$[B,T]_A = 0$ for $T = \gam[n-1]{A}{B}(B)$. Finally, since $A = K \vee B$,
\[
[A,T]_A \;=\; [K,T]_A \vee [B,T]_A \;=\; [K,T]_A
\;\leq\; [K,\gam[n-1]{A}{A}(A)]_A \;\leq\; \gam[n]{A}{A}(K) \;=\; 0 ,
\]
where we used $T \leq \gam[n-1]{A}{A}(A)$ and Lemma~\ref{lem:A} once more.
\end{proof}
\begin{proposition}\label{prop:n-nil}
Let $K$ and $B$ be subobjects of an object $A$ in a semi-abelian algebraically coherent category,
 with $K$ normal in $A$ and $K\vee B=A$, and let $n \geq 1$. If
\[\gam[n]{A}{A}(K) \;=\; 0 \]
and $B$ is $(n-1)$-nilpotent, then $A$ is $n$-nilpotent.
\end{proposition}
\begin{proof}
By Lemmas~\ref{lem:B} and~\ref{lem:C} we have
\[\gam[n]{A}{A}(A) \leq \gam[n]{A}{A}(K) \vee [A,\gam[n-1]{A}{B}(B)]=0.\qedhere\]
\end{proof}

\begin{remark}
For $n=2$, Proposition~\ref{prop:n-nil} says that if $A = K \vee B$ with $K$ normal, $B$ abelian and $[A,[A,K]] = 0$, then $A$ is $2$-nilpotent: this is exactly the computation carried out in the proof of Theorem~\ref{thm:main}. For $n = 1$ the hypothesis on $B$ means $B = 0$, so that $A = K$ and the statement is trivial. In the category of groups, Proposition~\ref{prop:n-nil} is an easy consequence of the Three Subgroups Lemma; its validity in the present categorical setting is what allows the inductive argument on $n$.
\end{remark}

The following lemma extends Lemma~\ref{lem:pullback}; as in that case, no hypothesis beyond semi-abelianness is needed.
\begin{lemma}\label{lem:pullback-n}
Let $f \colon A \to B$ and $p \colon E \to B$ be extensions in a semi-abelian category $\C$, let $n \geq 1$, and
consider the pullback of $f$ along $p$:
\[
\begin{tikzcd}[ampersand replacement=\&, column sep=large, row sep=large]
P \arrow[r, "{\pi_2}"] \arrow[d, "{\pi_1}"'] \& A \arrow[d, "{f}"] \\
E \arrow[r, "{p}"'] \& B .
\end{tikzcd}
\]
Write $K = \Ker(f)$ and identify, as usual, $\Ker(\pi_1)$ with $K$.
Then
\[
  \gam[n]{A}{A}(K) = 0
  \qquad \Longleftrightarrow \qquad
  \gam[n]{P}{P}(K) = 0 .
\]
\end{lemma}
\begin{proof}
Write $L = \Ker(\pi_2)$, so that $L$ is
canonically isomorphic to $\Ker(p)$. Since $p$ is a regular epimorphism,
so is its pullback $\pi_2$; and the restriction of $\pi_2$ to
$\Ker(\pi_1)$ is precisely the inclusion $K \to A$, whence
$\pi_2(K) = K$. Moreover $K \wedge L = 0$,  since
$K \wedge L = \Ker(\pi_1) \wedge \Ker(\pi_2)$, with $\pi_1$ and $\pi_2$ jointly monomorphic.

Since regular epimorphisms preserve Huq commutators, an immediate induction on $n$ shows that
\[
  \pi_2\bigl(\gam[n]{P}{P}(K)\bigr)
    \;=\; \gam[n]{\pi_2(P)}{\pi_2(P)}(\pi_2(K))
    \;=\; \gam[n]{A}{A}(K) .
\]
In particular, $\gam[n]{P}{P}(K)=0$ implies $\gam[n]{A}{A}(K) = 0$.

Conversely, suppose that $\gam[n]{A}{A}(K) = 0$.
Since $\gam[n]{P}{P}(K) = [P,\gam[n-1]{P}{P}(K)] \leq P \wedge \gam[n-1]{P}{P}(K) = \gam[n-1]{P}{P}(K)$, an induction on $n$ shows that $\gam[n]{P}{P}(K) \leq K$.
On the other hand, by the displayed equality above,
\[
\pi_2\bigl(\gam[n]{P}{P}(K)\bigr)
    \;= \gam[n]{A}{A}(K) \;=\; 0,
\]
so that $\gam[n]{P}{P}(K) \leq \Ker(\pi_2) = L$. Therefore
$\gam[n]{P}{P}(K) \leq K \wedge L = 0$, as desired.
\end{proof}

We shall also need the Smith--Pedicchio analogue of the iterated Huq
commutators $\gam[k]{A}{A}(A)$. For an object $A$ and $k \geq 0$ we define
inductively the equivalence relations $\nabk{k}$ on $A$ by
\[
  \nabk{0} \;=\; \nabla_A
  \qquad\text{and}\qquad
  \nabk{k+1} \;=\; \bigl[\nabla_A,\, \nabk{k}\bigr]_{\mathsf{Smith}} ,
\]
so that
$\nabk{k} = [\nabla_A,[\nabla_A,\dots,[\nabla_A,\nabla_A]_{\mathsf{Smith}}\dots]_{\mathsf{Smith}}]_{\mathsf{Smith}}$
with $k$ nested Smith commutators; thus $\nabk{1} = [\nabla_A,\nabla_A]_{\mathsf{Smith}}$
is the relation used in the proof of Theorem~\ref{thm:main}. Under the condition (SH)
the normalization of a Smith commutator is the Huq commutator of the
normalizations, $I_{[R,S]_{\mathsf{Smith}}} = [I_R,I_S]$ (see Section~\ref{ssec:smith}); since
$I_{\nabla_A} = A$, an immediate induction on $k$ gives
\[
  I_{\nabk{k}} \;=\; \gam[k]{A}{A}(A)
  \qquad (k \geq 0),
\]
and consequently $A/\nabk{k} \cong A/\gam[k]{A}{A}(A)$. In particular, in an
algebraically coherent semi-abelian category the normalization of $\nabk{n-1}$
is $\gam[n-1]{A}{A}(A)$, the kernel of the unit $\eta_A \colon A \to A/\gam[n-1]{A}{A}(A)$
of the reflection onto $\Nil_{n-1}(\C)$.

With the above results in place, our proof for the $2$-nilpotent case goes through essentially verbatim.

\begin{proof}[Proof of Theorem~\ref{thm:n-main}]
Assume that $\gam[n]{A}{A}(K) = 0$. By Lemma \ref{lem:A} (applied with $n-1$ in place of $n$) we obtain
\[
  [K,\gam[n-1]{A}{A}(A)] \;=\; 0 .
\]
Now the normalization of the equivalence relation $\Eq(f)$ on $A$ is $K$,
and the normalization of $\nabk{n-1}$ is $\gam[n-1]{A}{A}(A)$; the condition
(SH), which holds since $\C$ is algebraically coherent, therefore turns
the vanishing of the Huq commutator $[K,\gam[n-1]{A}{A}(A)]$ into
the vanishing of the Smith commutator
\[
  \bigl[\,\Eq(f),\; \nabk{n-1}\,\bigr]_{\mathsf{Smith}} \;=\; \Delta_A .
\]
Accordingly, the equivalence relations $\Eq(f)$ and $\nabk{n-1}$
centralize each other, and there is a unique double centralizing
equivalence relation $C$ on them. Taking the quotients in the vertical
direction, we obtain the diagram
\begin{equation}\label{diag:key2}
\begin{tikzcd}[ampersand replacement=\&, column sep=huge, row sep=large]
C
  \arrow[r, shift left=0.9ex, "{p_1}"] \arrow[r, shift right=0.9ex, "{p_2}"']
  \arrow[d, shift left=0.9ex, "{c_1}"] \arrow[d, shift right=0.9ex, "{c_2}"']
\& {\nabk{n-1}}
  \arrow[d, shift left=0.9ex, "{d_1}"] \arrow[d, shift right=0.9ex, "{d_2}"'] \\
\Eq(f)
  \arrow[r, shift left=0.9ex, "{f_1}"] \arrow[r, shift right=0.9ex, "{f_2}"']
  \arrow[d, two heads, "{\psi}"']
\& A
  \arrow[d, two heads, "{\varphi}"] \\
\Eq(f)/C
  \arrow[r, shift left=0.9ex, "{\bar{f}_1}"]
  \arrow[r, shift right=0.9ex, "{\bar{f}_2}"']
\& A/\nabk{n-1}
\end{tikzcd}
\end{equation}
which is obtained in perfect analogy with the construction of diagram \eqref{diag:key}.
The
morphisms $\psi$ and $\varphi$ are the quotients of the two vertical
equivalence relations, the latter having codomain
$A/\nabk{n-1} \cong A/\gam[n-1]{A}{A}(A)$. To lighten the
notation we shall write
\[
  E \;=\; \Eq(f)/C \qquad\text{and}\qquad Q \;=\; A/\nabk{n-1}
\]
for the quotients appearing in the lower left and right corner
of~\eqref{diag:key2}, respectively.

As in the proof of Theorem \ref{thm:main}, the lower part of~\eqref{diag:key2} is then a discrete fibration of internal
groupoids
\[
  \bigl(\Eq(f) \rightrightarrows A\bigr)
  \longrightarrow
  \bigl(E \rightrightarrows Q\bigr) .
\]
The next step is to show that $E$ is $n$-nilpotent. Since $\Eq(f)$ is
obtained by pulling back $f$ along itself, Lemma \ref{lem:pullback-n} tells us that
$\gam[n]{\Eq(f)}{\Eq(f)}(K)=0$. Since each of the lower commutative squares of the diagram \eqref{diag:key2} is a pullback, the kernel of $\bar{f}_1$ is $K$ and $\psi(K)=K$; as $\psi$ is a regular epimorphism, it follows that $\gam[n]{E}{E}(K)= \psi\bigl(\gam[n]{\Eq(f)}{\Eq(f)}(K)\bigr)=0$. Since $\overline{f}_1$ is a split epimorphism,
we can view $Q$ as a subobject of $E$, and then $E=K\vee Q$. Moreover, $Q \cong A/\gam[n-1]{A}{A}(A)$ is $(n-1)$-nilpotent, since
\[
  \gam[n-1]{Q}{Q}(Q) \;=\; \varphi\bigl(\gam[n-1]{A}{A}(A)\bigr) \;=\; 0 .
\]
Proposition \ref{prop:n-nil} now tells us that $E$ is $n$-nilpotent.
Writing $I \colon \C \rightarrow \Nil_{n}(\C)$ for the reflector onto the subcategory of $n$-nilpotent objects
and $\eta$ for its unit, via the universal property of the reflector there are unique morphisms $u$ and $v$ making the diagram
\[
\begin{tikzcd}[ampersand replacement=\&, column sep=huge, row sep=large]
\Eq(f) \arrow[bend left=20,rr, two heads, "{\psi}"] \arrow[r, two heads, "{\eta_{\Eq(f)}}"] \arrow[d, "{f_1}"'] \& I(\Eq(f)) \arrow[r, two heads, "{u}"] \arrow[d, "{I(f_1)}"'] \& E
  \arrow[d, "{\bar{f}_1}"] \\
A \arrow[bend right=20,rr, two heads, "{\varphi}"']\arrow[r, two heads, "{\eta_{A}}"'] \& I(A)\arrow[r, two heads, "{v}"'] \& Q
\end{tikzcd}
\]
commute. As in the proof of Theorem \ref{thm:main}, from this it follows that the left hand square is a pullback, so that $f_1$ is a trivial extension and $f$ is a normal, hence central, extension.

Conversely, suppose that $f$ is central, hence normal. Then $\eta_{\Eq(f)}$ and $f_1$ are jointly monomorphic, so that $\Ker(\eta_{\Eq(f)}) \wedge K = 0$, where now $$\Ker(\eta_{\Eq(f)}) = \gam[n]{\Eq(f)}{\Eq(f)}(\Eq(f)).$$ Since $\gam[n]{\Eq(f)}{\Eq(f)}(K) \leq \gam[n]{\Eq(f)}{\Eq(f)}(\Eq(f)) \wedge K$ by monotonicity and normality, we get $\gam[n]{\Eq(f)}{\Eq(f)}(K) = 0$, and Lemma~\ref{lem:pullback-n} gives $\gam[n]{A}{A}(K) = 0$.
\end{proof}

\section{The characterization via the Higgins commutator}\label{sec:higgins}

In this last section $\C$ denotes an \emph{arbitrary} semi-abelian category: algebraic coherence is not assumed. We show that, when the nilpotent objects are defined by means of the higher-order Higgins commutators, the relative central extensions can still be characterized by a commutator condition.

We first recall the definition of the Higgins commutator. This notion goes back to the work of Higgins on $\Omega$-groups \cite{Hig56}, and it was introduced in the categorical context by Mantovani and Metere \cite{MM10}; its higher-order versions were introduced by Hartl and Loiseau \cite{HL13} and further studied by Hartl and Van der Linden \cite{HVdL}. For objects $K$ and $L$ of $\C$, the \emph{co-smash product} $K \diamond L$ is the kernel of the canonical morphism $K + L \to K \times L$ from the coproduct to the product. Given subobjects $k \colon K \to X$ and $l \colon L \to X$ of an object $X$, their \emph{Higgins commutator} $[K,L]_{\mathsf{Hig}}$ is the subobject of $X$ obtained as the regular image of the composite
\[
\begin{tikzcd}[ampersand replacement=\&, column sep=large]
K \diamond L \arrow[r, tail] \& K + L \arrow[r, "{[k,l]}"] \& X .
\end{tikzcd}
\]
The Higgins commutator $[K,L]_{\mathsf{Hig}}$ need not be normal in $X$, and its normal closure in $X$ is the Huq commutator $[K,L]$ \cite[Proposition~5.7]{MM10}; in particular, $[K,L]_{\mathsf{Hig}} = 0$ if and only if $[K,L] = 0$. Similarly, the \emph{ternary co-smash product} $K \diamond L \diamond M$ of three objects is the kernel of the morphism
\[
  K + L + M \longrightarrow (K+L) \times (K+M) \times (L+M)
\]
whose three components are the morphisms $K+L+M \to K+L$, $K+L+M \to K+M$ and $K+L+M \to L+M$ which are the identity on two of the summands and zero on the third one. Given subobjects $k \colon K \to X$, $l \colon L \to X$ and $m \colon M \to X$ of $X$, the \emph{ternary Higgins commutator} $[K,L,M]_{\mathsf{Hig}}$ is the subobject of $X$ obtained as the regular image of the composite
\[
\begin{tikzcd}[ampersand replacement=\&, column sep=large]
K \diamond L \diamond M \arrow[r, tail] \& K + L + M \arrow[r, "{[k,l,m]}"] \& X .
\end{tikzcd}
\]
The $n$-ary co-smash product $K_1 \diamond \dots \diamond K_n$ and the $n$-ary Higgins commutator $[K_1, \dots, K_n]_{\mathsf{Hig}}$ of subobjects $K_1, \dots, K_n$ of $X$ are defined in the same way, for any $n \geq 2$; we refer to \cite{HVdL} for the details. Observe that, being defined as an image, a Higgins commutator does not depend on the ambient object: if $K_1, \dots, K_n$ are subobjects of $Y$ and $Y$ is a subobject of $X$, then $[K_1, \dots, K_n]_{\mathsf{Hig}}$ is the same subobject of $X$ whether it is computed in $Y$ or in $X$. In particular $[K_1, \dots, K_n]_{\mathsf{Hig}} \leq Y$.

We shall need the following properties of the Higgins commutators of subobjects $K_1, \dots, K_n$ and $L$ of an object $X$, established in \cite[Propositions~2.21 and~2.22]{HVdL}:
\begin{itemize}
\item[(H1)] if $K_i = 0$ for some $i$, then $[K_1, \dots, K_n]_{\mathsf{Hig}} = 0$;
\item[(H2)] $[K_1, \dots, K_n]_{\mathsf{Hig}} = [K_{\sigma(1)}, \dots, K_{\sigma(n)}]_{\mathsf{Hig}}$ for any permutation $\sigma$ of $\{1, \dots, n\}$;
\item[(H3)] $g([K_1, \dots, K_n]_{\mathsf{Hig}}) = [g(K_1), \dots, g(K_n)]_{\mathsf{Hig}}$ for any regular epimorphism $g \colon X \to Y$;
\item[(H4)] if $L \leq K_i$, then $[K_1, \dots, K_{i-1}, L, K_{i+1}, \dots, K_n]_{\mathsf{Hig}} \leq [K_1, \dots, K_n]_{\mathsf{Hig}}$;
\item[(H5)] if $K_i = K_{i+1}$, then
\[
  [K_1, \dots, K_i, K_{i+1}, \dots, K_n]_{\mathsf{Hig}} \;\leq\; [K_1, \dots, K_{i}, K_{i+2}, \dots, K_n]_{\mathsf{Hig}} ;
\]
\item[(H6)] the following equality holds:
\begin{align*}
  [K_1, \dots, K_{n-1}, K_n \vee L]_{\mathsf{Hig}} \;=\;{} & [K_1, \dots, K_{n-1}, K_n]_{\mathsf{Hig}} \vee [K_1, \dots, K_{n-1}, L]_{\mathsf{Hig}} \\
  & \vee [K_1, \dots, K_{n-1}, K_n, L]_{\mathsf{Hig}} .
\end{align*}
\end{itemize}

An object $X$ of $\C$ is called \emph{two-nilpotent} \cite{CGVdL, DHVdL} when $[X,X,X]_{\mathsf{Hig}} = 0$. The full subcategory $\Nil^{\mathsf{H}}_2(\C)$ of $\C$ determined by the two-nilpotent objects, called the \emph{two-nilpotent core} of $\C$, is a Birkhoff subcategory of $\C$, and the $X$-component of the unit of the corresponding reflection is the canonical quotient $\eta_X \colon X \to X/[X,X,X]_{\mathsf{Hig}}$ \cite{CGVdL, DHVdL}. When $\C$ is algebraically coherent one has $[X,X,X]_{\mathsf{Hig}} = [[X,X],X]$ for every object $X$ \cite[Corollary~7.2]{CGVdL}, so that in that case $\Nil^{\mathsf{H}}_2(\C)$ is the subcategory $\Nil_2(\C)$ considered in the previous sections.

\begin{theorem}\label{thm:higgins}
Let $\C$ be a semi-abelian category, and let $f \colon A \to B$ be an extension in $\C$ with kernel $K$. Then $f$ is a central extension relative to the Birkhoff subcategory $\Nil^{\mathsf{H}}_2(\C)$ if and only if
\[
  [K,A,A]_{\mathsf{Hig}} \;=\; 0 ,
\]
where $[\,\cdot\,,\cdot\,,\cdot\,]_{\mathsf{Hig}}$ denotes the ternary Higgins commutator.
\end{theorem}

\begin{proof}
Write $R = \Eq(f)$, with projections $f_1, f_2 \colon R \to A$ and diagonal $\delta \colon A \to R$, and let $D = \delta(A)$ be the corresponding subobject of $R$. As in Lemma~\ref{lem:pullback} we identify $K$ with $\Ker(f_1)$, so that $f_1(K) = 0$ and $f_2(K) = K$; moreover $f_1(R) = A = f_2(R)$. Recall from Section~\ref{sec:galois} that $f$ is central if and only if it is normal, that is, if and only if $f_1$ is a trivial extension. Writing $I \colon \C \to \Nil^{\mathsf{H}}_2(\C)$ for the reflector, this means that the naturality square
\begin{equation}\label{diag:higgins-square}
\begin{tikzcd}[ampersand replacement=\&, column sep=huge, row sep=large]
R \arrow[r, two heads, "{\eta_R}"] \arrow[d, two heads, "{f_1}"'] \& I(R) = R/[R,R,R]_{\mathsf{Hig}} \arrow[d, two heads, "{I(f_1)}"] \\
A \arrow[r, two heads, "{\eta_A}"'] \& I(A) = A/[A,A,A]_{\mathsf{Hig}}
\end{tikzcd}
\end{equation}
is a pullback. Since \eqref{diag:higgins-square} is a regular pushout, as recalled after Definition~\ref{def:trivial}, it is a pullback precisely when $\eta_R$ and $f_1$ are jointly monomorphic, that is, when $\Ker(\eta_R) \wedge \Ker(f_1) = 0$. In other words, $f$ is central if and only if
\begin{equation}\label{eq:higgins-central}
  [R,R,R]_{\mathsf{Hig}} \wedge K \;=\; 0 .
\end{equation}

Suppose first that $f$ is central, so that \eqref{eq:higgins-central} holds. On the one hand $[K,R,R]_{\mathsf{Hig}} \leq [R,R,R]_{\mathsf{Hig}}$ by (H4). On the other hand, by (H3) and (H1),
\[
  f_1\bigl([K,R,R]_{\mathsf{Hig}}\bigr) \;=\; [f_1(K), f_1(R), f_1(R)]_{\mathsf{Hig}} \;=\; [0,A,A]_{\mathsf{Hig}} \;=\; 0 ,
\]
so that $[K,R,R]_{\mathsf{Hig}} \leq \Ker(f_1) = K$. Hence $[K,R,R]_{\mathsf{Hig}} \leq [R,R,R]_{\mathsf{Hig}} \wedge K = 0$, and applying (H3) to $f_2$ we conclude that
\[
  [K,A,A]_{\mathsf{Hig}} \;=\; [f_2(K), f_2(R), f_2(R)]_{\mathsf{Hig}} \;=\; f_2\bigl([K,R,R]_{\mathsf{Hig}}\bigr) \;=\; 0 .
\]

Conversely, suppose that $[K,A,A]_{\mathsf{Hig}} = 0$. By (H3) and (H1) we have
\begin{align*}
  f_1\bigl([K,R,R]_{\mathsf{Hig}}\bigr) &= [0,A,A]_{\mathsf{Hig}} = 0 , \\
  f_2\bigl([K,R,R]_{\mathsf{Hig}}\bigr) &= [K,A,A]_{\mathsf{Hig}} = 0 ,
\end{align*}
and then $[K,R,R]_{\mathsf{Hig}} = 0$, because $f_1$ and $f_2$ are jointly monomorphic. By (H2), (H4) and (H5) it follows that any ternary or quaternary Higgins commutator of subobjects of $R$ having $K$ among its entries is trivial, since such a commutator is contained in $[K,R,R]_{\mathsf{Hig}}$ or in $[K,R,R,R]_{\mathsf{Hig}} \leq [K,R,R]_{\mathsf{Hig}}$. Now $f_1$ is a split epimorphism with section $\delta$ and kernel $K$, so that $R = K \vee D$. Using (H2) and (H6) we then get
\begin{align*}
[R,R,R]_{\mathsf{Hig}} &= [K \vee D, R, R]_{\mathsf{Hig}} \\ &= [K,R,R]_{\mathsf{Hig}} \vee [D,R,R]_{\mathsf{Hig}} \vee [K,D,R,R]_{\mathsf{Hig}} = [D,R,R]_{\mathsf{Hig}}, \\
[D,R,R]_{\mathsf{Hig}} &= [D, K \vee D, R]_{\mathsf{Hig}} \\ &= [D,K,R]_{\mathsf{Hig}} \vee [D,D,R]_{\mathsf{Hig}} \vee [D,K,D,R]_{\mathsf{Hig}} = [D,D,R]_{\mathsf{Hig}}, \\
[D,D,R]_{\mathsf{Hig}} &= [D, D, K \vee D]_{\mathsf{Hig}} \\ &= [D,D,K]_{\mathsf{Hig}} \vee [D,D,D]_{\mathsf{Hig}} \vee [D,D,K,D]_{\mathsf{Hig}} = [D,D,D]_{\mathsf{Hig}},
\end{align*}
and therefore $[R,R,R]_{\mathsf{Hig}} = [D,D,D]_{\mathsf{Hig}} \leq D$. Since $f_1 \circ \delta = 1_A$, the restriction of $f_1$ to $D$ is a monomorphism, whence $D \wedge K = 0$. Since $[R,R,R]_{\mathsf{Hig}} \leq D$, as we have just seen, taking the meet with $K$ on both sides of this inequality gives $[R,R,R]_{\mathsf{Hig}} \wedge K \leq D \wedge K = 0$, which is \eqref{eq:higgins-central}, and $f$ is central.
\end{proof}

\begin{remark}\label{rem:higgins-coherent}
When $\C$ is, moreover, algebraically coherent, Theorem~\ref{thm:higgins} gives back Theorem~\ref{thm:main}. Indeed, in that case $\Nil^{\mathsf{H}}_2(\C) = \Nil_2(\C)$, as recalled above, while Theorem~7.1 in \cite{CGVdL}, applied to the normal subobjects $K$, $A$ and $A$ of $A$, gives
\[
  [K,A,A]_{\mathsf{Hig}} \;=\; [[K,A]_{\mathsf{Hig}},A]_{\mathsf{Hig}} \vee [[A,K]_{\mathsf{Hig}},A]_{\mathsf{Hig}} \;=\; [A,[K,A]] ,
\]
where the last equality holds because, in an algebraically coherent semi-abelian category, the Higgins commutator of two normal subobjects is normal, hence it coincides with their Huq commutator \cite[Theorem~6.18]{CGVdL}. On the other hand Theorem~\ref{thm:higgins} also applies to semi-abelian categories which are not algebraically coherent, such as the category of skew braces \cite{SV}.
\end{remark}

The same argument can be extended to the $n$-nilpotent case, for any $n \geq 1$. An object $X$ of $\C$ is \emph{$n$-nilpotent} when the $(n+1)$-ary Higgins commutator $[X, \dots, X]_{\mathsf{Hig}}$ vanishes, and these objects determine a Birkhoff subcategory $\Nil^{\mathsf{H}}_n(\C)$ of $\C$, the reflector sending an object $X$ to the quotient $X/[X,\dots,X]_{\mathsf{Hig}}$ \cite{DHVdL}. By using the properties (H1)--(H6) of the $(n+1)$-ary and of the $(n+2)$-ary Higgins commutators exactly as in the proof of Theorem~\ref{thm:higgins}, one obtains the following result, whose proof we omit:

\begin{theorem}\label{thm:higgins-n}
Let $\C$ be a semi-abelian category, let $n \geq 1$, and let $f \colon A \to B$ be an extension in $\C$ with kernel $K$. Then $f$ is a central extension relative to the Birkhoff subcategory $\Nil^{\mathsf{H}}_n(\C)$ if and only if
\[
  [K,\underbrace{A,\dots,A}_{\text{$n$ terms}}]_{\mathsf{Hig}} \;=\; 0 ,
\]
where $[\,\cdot\,,\dots,\cdot\,]_{\mathsf{Hig}}$ denotes the $(n+1)$-ary Higgins commutator.
\end{theorem}

For $n = 1$ this is again the classical characterization of the $\Ab(\C)$-central extensions \cite{BG02-central}, since $[K,A]_{\mathsf{Hig}} = 0$ if and only if $[K,A] = 0$.

\end{document}